\documentclass[10pt,runningheads,a4paper]{llncs}

\usepackage{graphicx}
\usepackage{amsmath}
\usepackage{amssymb}
\usepackage{caption}
\numberwithin{theorem}{section} \numberwithin{lemma}{section} \numberwithin{corollary}{section} \numberwithin{definition}{section} \numberwithin{remark}{section} \numberwithin{corollary}{section} \numberwithin{example}{section} \numberwithin{proposition}{section}

\begin{document}

\title{Order divisor graphs of finite groups}
\author{S. U. Rehman$^1$, A. Q. Baig$^1$, M. Imran $^2$, and Z. U.  Khan$^1$}

\vspace{0.4cm}

\institute{$^1$Department of Mathematics,\\
COMSATS Institute of Information  and Technology,\\
Attock, Pakistan.\\
$^2$ Department of Mathematical Sciences,\\
United Arab Emirates University,\\ P.O. Box 15551, Al Ain, United Arab Emirates.\\
$^2$Department of Mathematics,\\
School of Natural Sciences(SNS),\\
National University of Sciences and Technology (NUST),\\
Sector H-12, Islamabad, Pakistan.
\\
E-mail: \{aqbaig1, imrandhab, zizu006\}@gmail.com\\
shafiq\_ur\_rahman2@yahoo.com}

\vspace{0.4cm}

\authorrunning{S. U. Rehman, A. Q. Baig, M. Imran, Z. U.  Khan}
\titlerunning{Order divisor graphs of finite groups} \maketitle

\markboth{\small{S. U. Rehman, A. Q. Baig, M. Imran, Z. U.  Khan}} {\small{Order divisor graphs of finite groups}}

\begin{abstract}
The interplay between groups and graphs have
been  the most famous and productive area of algebraic graph theory. In this paper, we introduce and study the graphs whose
vertex set is  group $G$ such that two distinct vertices $a$ and $b$ having
different orders are adjacent provided that $o(a)$ divides $o(b)$ or $o(b)$ divides $o(a)$.
\end{abstract}

\centerline{{\bf 2010 Mathematics Subject Classification:} 05C25}
\vspace{.3cm} {\bf Keywords:} Finite groups, Cyclic groups,  p-groups, Elementary abelian groups,  Connected Graphs,
Star Graphs, Sequential Joins.

\section[Introduction]{Introduction}

                       All finite groups can be represented as the automorphism group of a connected graph \cite{F}. A graphical representation of group can be given by a set of generators and relations. Given any group, symmetrical graphs known as Cayley graphs can be generated, and these have properties related to the structure of the group \cite{G}. Relating a graph to a group  provides a method of visualizing a group and connects two important branches of mathematics. It gives a review of cyclic groups, dihedral groups, direct products,  generators and relations. \\

   Groups are the main mathematical tools for studying symmetries of an object and symmetries are usually
related to graph automorphisms. Many structures in abstract algebra are special cases of groups.  Graph theory is one of the leading research field in mathematics mainly because of its applications in diverse fields which include biochemistry, electrical engineering, computer science and operations research. These both branches of mathematics are playing a vital role in modern mathematics. In group theory we study and analyze different groups and their structures while in graph theory we focus on the graphs that denotes the structure of materials and objects. The powerful combinatorial methods found in graph theory have also been used to prove significant and well-known results in a variety of areas in mathematics including group theory.  \\

In the last few decades the researchers have focused on the modified form of groups and graphs by inter relating their properties. The study of the algebraic structures using the properties of graphs has become an inspiring research topic in the last twenty years, leading to many fascinating results and questions, see \cite{AR2}, \cite{AR1}, \cite{R}, and \cite{Re}. In this paper, we study a graph related to finite groups. We call a graph  an {\em order divisor graph}, denoted by $OD(G)$,  if its vertex set is a finite group $G$ and two distinct vertices $a$ and $b$ having different orders are adjacent provided that $o(a)$ divides $o(b)$ or $o(b)$ divides $o(a)$.\\

For the
reader's convenience we give a working introduction here for the notions
involved. A group $G$ is called a {\em $p$-group} if every element of $G$ has order a power of $p$, where $p$ is prime.  The {\em exponent of a group} $G$ is the smallest positive integer $m$ such that $g^m=e$ for all $g \in G$. An abelian group $G$ is called an {\em elementary abelin group} or {\em elementary abelin $p$-group} if  it is a $p$-group of exponent $p$ for some prime $p$, i.e., $x^p=e$~ $ \forall ~x \in G$. A  finite {\em elementary abelin group} is a group that is isomorphic to $\mathbb{Z}_p^n$ for some prime $p$ and for some positive integer $n$. The  group of symmetries of a regular $n$-gon ($n \geq 3$) is called a {\em dihedral group} or order $2n$. The dihedral group has presentation $D_n=<a,b \mid a^n=b^2=(ab)^2=e>$. We denote by $U(\mathbb{Z}_n)$, the group of units of $\mathbb{Z}_n$, i.e., $U(\mathbb{Z}_n)=\{ \bar{x} \in \mathbb{Z}_n \mid (x,n)=1\}$.  For any two elements $a,b$ of a group $G$, $[a,b]$ denote the {\em commutator} $a^{-1}b^{-1}ab$. If $A,B$ are subsets of a group $G$ then $[A,B]=<[a,b] \mid a \in A, b \in B>$. Particularly, $[G,G]$ is called {\em commutator subgroup} denoted by $G'$. A group $G$ is called {\em nilpotent group} if and only if it is the direct product of its Sylow subgroups, cf. \cite[Chapter 2, Section 3]{G2}. The unique maximal nilpotent normal subgroup of a group $G$ is called {\em Fitting subgroup} of $G$ and is denoted by $F(G)$, cf. \cite[Chapter 6, Section 1]{G2}. A simple connected graph is an undirected graph without any loops and multiple edges. A complete bipartite graph is a bipartite graph (i.e., a set of graph vertices decomposed into two disjoint sets such that no two graph vertices within the same set are adjacent) such that every pair of graph vertices in the two sets are adjacent. The star graph $S_n$ of order $n$ is a tree with one vertex of degree $n-1$ and all other vertices have degree $1$, i.e., $S_n \cong K_{1,n}$. We denote by $G_1 \diamond G_2 \diamond \cdots \diamond G_k$  the sequential join of graphs $G_1,G_2,...,G_k$, where $G_i \diamond G_{i+1}=G_i \vee G_{i+1}$ for all $1 \leq i \leq k-1$, i.e., by adding an edge from each vertex of $G_i$ to each vertex of $G_{i+1}$, $1 \leq i \leq k-1$ The chromatic number  $\chi(G)$ of a graph $G$ is defined to be the minimum number of colors required to color the vertices of $G$,  i.e., $\chi(G)=min\{k:G~is~k-colorbale\}$.\\

In this paper we obtain the following results. The order divisor graph $OD(G)$ of a finite group $G$ is a star graph if and only if every non-identity element of $G$ has prime order (Theorem \ref{9}). For an abelian group $G$, $OD(G)$ is a star graph if and only if $G$ is elementary abelian (Corollary \ref{13}). The order divisor graph $OD\big(U(\mathbb{Z}_n)\big)$ is a star graph $S_{\phi(n)}$ if and only if $n \mid 24$ (Corollary \ref{11}). The order divisor graph $OD(\mathbb{Z}_n)$ is a star graph if and only if $n$ is prime (Corollary \ref{12}). The order divisor graph $OD(G)$ of a (finite) group $G$ is a star graph if and only if $G$ is a $p$-group of exponent $p$, or a non-nilpotent group of order $p^aq$, or it is
isomorphic to the simple group $\mathcal{A}_5$ (Corollary \ref{10}). The order divisor graph of the dihedral group  $D_n$ $(n \geq 3)$ is a star graph $S_{2n}$ if and only if $n$ is  prime (Theorem \ref{4}). If $G$ is a finite $p$-group of order $p^n$ then
$OD(G)$ is a complete multi-partite graph (Theorem \ref{2}). If $G$ is a finite cyclic group of order $p^n$ then $OD(G)$ is complete $(n+1)$-partite graph (Theorem \ref{3}). If $G$ is a finite cyclic group of order $p^n$, then $\chi(OD(G))=n+1$ (Corollary \ref{6}).  If $G$ is a cyclic group of order $p_1p_2$, where $p_1, p_2$ are distinct primes, then $OD(G)$ is a sequential join of graphs (Theorem \ref{5}). Similarly, if $G$ is a cyclic group of order $p_1p_2p_3$, where $p_1, p_2, p_3$ are distinct primes, then $OD(G)$ is obtained by certain type of  sequential and cyclic joins. (Theorem \ref{7}).  Let $n \in \mathbb{N}$ and let $D$ be the set of all $(positive)$ divisors of $n$. Define a partial order $\preceq$ on $D$ by $a \preceq b$ if and only if $a \mid b$. Then $(D, \preceq)$ is a bounded lattice. We denote by $G_n$  the comparability graph on $(D, \preceq)$. In other words, $G_n$ is a simple undirected graph with vertex set $D$ and two vertices $a$ and $b$ are
adjacent if and only if $a \neq  b$ and either $a \preceq b$ or $b \preceq a$. The new extended graph is denoted by $\mathcal{E}(G_n)$.  Given a graph $G= (V, E)$, the reduced graph of $G$, denoted by $\mathcal{R}(G)$,   is obtained from $G$ by merging those vertices which has same
set of closed neighbors. Note that a closed neighbor of $v \in V$ is the set $\{u \in  V | uv \in E\} \cup\{v\}$. We obtain that if $G$ is finite group of order $n$, then $G$ is cyclic if and only if $\mathcal{E}(G_n) \cong OD(G)$
if and only if  $G_n \cong \mathcal{R}(OD(G))$ (Theorem \ref{8}).\\

In this paper all the groups and graphs discussed are finite.  We follow the terminologies and notations of \cite{G} for groups and \cite{W} for graphs.

\section{Main results}

\begin{definition}
Let $G$ be a finite group. Then $OD(G)$ denotes the order divisor graph whose
vertex set is $G$ such that two distinct vertices $a$ and $b$ having
different orders are adjacent provided that $o(a)  \mid  o(b)$ or
$o(b) \mid o(a)$.
\end{definition}

\begin{remark}Some easy consequences of the definition are:

\begin{enumerate}

\item[(i)]  $OD(G)$ is a simple graphs, so there are no loops and multiple edges.

\item[(ii)] Since the identity element is the only element of a group having order one, so the vertex associated to the identity element
is adjacent to each vertex and hence $OD(G)$ is always a connected graph. Due to similar reason, if $|G| >2$, then $OD(G)$ has diameter $2$.

\item[(iii)]  Since the vertex associated to the identity
is adjacent to each vertex, therefore,  if $|G| >3$ then $OD(G)$ is not a cycle. If $|G|=3$ then $G$ has two elements of order $3$ and the vertices associated to these two elements are not adjacent. Hence $OD(G)$ cannot be a cycle.

    \item[(iv)] If $G$ is finite,  then for every divisor $d$ of $|G|$, the number of elements of order $d$ is a multiple of $\phi(d)$ ($\phi$ is Euler's phi function).  Hence, if $|G| > 2$, then $OD(G)$ cannot be a complete graph.

\end{enumerate}
\end{remark}

\begin{theorem}\label{9}
The order divisor graph $OD(G)$ is a star graph if and only if every non-identity element of the group $G$ has prime
order.
\end{theorem}
\begin{proof}

Clearly, if every non-identity element of $G$ has prime order, then $OD(G)$ is a star graph. Conversely, suppose that $OD(G)$ is a star graph and $x_1,x_2,...,x_n$ are distinct non-identity elements of $G$  with $o(x_i)=d_i$ for $1 \leq i \leq n$. As $OD(G)$ is a star graph, so $d_i \nmid d_j$ for all $i \neq j$. If $d_i$ is not prime for some $i$, then $d_i$ has a prime divisor, say $p$. But then by Cauchy's theorem, $G$ must have an element of order $p$. Thus $p=d_j$ for some $i \neq j$, which is a contradiction. Hence each  $d_i$ is prime.
\end{proof}

Recall \cite{G2} that an {\em elementary abelian group} or {\em elementary abelian $p$-group} is an abelian $p$-group of exponent $p$ for some fixed prime $p$. A finite {\em elementary abelian group} is a group that is isomorphic to $\mathbb{Z}_p^n$ for some prime $p$ and for some positive integer $n$.

\begin{corollary}\label{13}
Let $G$ be an abelian group. Then $OD(G)$ is a star graph if and only if $G$ is elementary abelian.
\end{corollary}
\begin{proof}
Suppose $OD(G)$ is a star graph. Then by Theorem \ref{9}, every non-identity element of  $G$ has prime order. Let $a,b \in G$ such that $o(a)$ and $o(b)$ are distinct primes.  Since $G$  is abelian,  so $o(ab)=o(a)o(b)$, which is a contradiction.  Hence $G$ is an elementary abelian group (abelian $p$-group of exponent $p$ for some prime number $p$).
\end{proof}

\begin{remark}\label{1}
If the group $G$ is not abelian, then above Corollary \ref{13} fails. For example, the order divisor graph $OD(D_3)$ of the dihedral group $D_3=<a,b$ $|$ $ a^3=b^2=(ab)^2=e>$ is a star graph but $D_3$ is not elementary abelian.
\end{remark}

\begin{figure}[h!]
   \centering
     \includegraphics[width=0.4\textwidth]{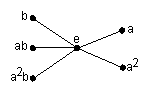}
     \caption{$OD(D_3)\cong S_6$}
\end{figure}

\begin{corollary}\label{11}
$OD\big(U(\mathbb{Z}_n)\big)$ is a star graph $S_{\phi(n)}$ if and only if $n \mid 24$.
\end{corollary}
\begin{proof}
Suppose $OD(U(\mathbb{Z}_n)) \cong S_{\phi(n)}$, where $\phi$ is the Euler's phi function. Then by Corollary \ref{13},  $U(\mathbb{Z}_n)$ is an elementary abelian group (abelian $p$-group of exponent $p$ for some prime number $p$).  Let $\phi(n) >1, i.e., n >2$ (If $\phi(n)=1$, then $U(\mathbb{Z}_n)$ is trivial). Then $U(\mathbb{Z}_n)$ has at least two elements that satisfy $x^2=e$. Hence every non-identity element of $U(\mathbb{Z}_n)$ has order $2$. Case 1: If $n$ is odd. Then $\bar{2} \in U(\mathbb{Z}_n)$ and so $n \mid 2^2-1=3$. Hence $n=1$ or $3$. Case 2: If $n=2^t$ for some $t \geq 1$. Then $ \bar{3} \in U(\mathbb{Z}_n)$ and so $n \mid 3^2-1=8$. This implies $n=1,2,4$ or $8$. Case 3: If $n$ is any arbitrary positive integer. Then $n=2^tk$, where $k$ is odd and $t$ is non-negative integer. Since $U(\mathbb{Z}_n) \cong U(\mathbb{Z}_{2^t}) \times U(\mathbb{Z}_k)$, so every non-identity element of both  $U(\mathbb{Z}_{2^t})$ and $U(\mathbb{Z}_{k})$ has order $2$.  By above cases $1$ and $2$, we get that  $2^t \in \{ 1,2,4,8\}$ and $k \in \{1,3\}$. Hence $n=2^tk \in \{1,2, 4, 6,8,12,24\}$.

Conversely, it is easy to check that if $n \mid 24$ then every  non-identity element of $U(\mathbb{Z}_n)$ has order $2$, i.e., $OD(U(\mathbb{Z}_n)) \cong S_{\phi(n)}$.
\end{proof}

\begin{corollary}\label{12}
$OD(\mathbb{Z}_n)$ is a star graph if and only if $n$ is prime.
\end{corollary}

Recall \cite{G2} that for any two elements $a,b$ of a group $G$, $[a,b]$ denote the {\em commutator} $a^{-1}b^{-1}ab$. If $A,B$ are subsets of a group $G$ then $[A,B]=<[a,b] \mid a \in A, b \in B>$. Particularly, $[G,G]$ is called {\em commutator subgroup} denoted by $G'$. A group $G$ is called {\em nilpotent group} if and only if it is the direct product of its Sylow subgroups. The unique maximal nilpotent normal subgroup of a group $G$ is called {\em Fitting subgroup} of $G$ and is denoted by $F(G)$.
\begin{corollary}\label{10}
The order divisor graph $OD(G)$ is a star graph if and only if one of the  following cases occur:
\begin{enumerate}
\item  $G$ is $p$-group of exponent $p$.
\item
\begin{enumerate}
\item[{\em(a)}] $|G|=p^aq,~ 3 \leq p < q,~ a \geq 3,~ |F(G)|=p^{a-1},~ |G:G'|=p$.
\item[{\em(b)}] $|G|=p^aq,~ 3 \leq p < q,~ a \geq 1,~ |F(G)|=|G'|=p^{a}$.
\item[{\em(c)}] $|G|=2^ap,~ p\geq 3,~ a \geq 2,~ |F(G)|=|G'|=2^{a}$.
\item[{\em(d)}] $|G|=2p^a,~ p\geq 3,~ a \geq 1,~ |F(G)|=|G'|=p^{a}$ and $F(G)$ is elementary abelian.
\end{enumerate}

\item  $G\cong \mathcal{A}_5$.

\end{enumerate}
\end{corollary}
\begin{proof}
Apply \cite[Main Theorem]{D}.
\end{proof}

\begin{theorem}\label{4}
 The order divisor graph of the dihedral group $D_n$ $(n \geq 3)$ is a star graph $S_{2n}$ if and only if $n$ is  prime.
\end{theorem}
\begin{proof}
$D_n=<a,b \mid a^n=b^2=(ab)^2=e>=\{e,a,a^2,..a^{n-1}, b, ab, a^2b,..., a^{n-1}b\}$.

Suppose $OD(D_n)$ is a star graph. So, every pair of vertices corresponding to non-identity elements is non-adjacent. Therefore, if $o(a^i) \neq o(a^j)$ for some $i,j \in \{1,2,3...n-1\}$, then $o(a^i) \nmid o(a^j)$ and $o(a^j) \nmid o(a^i)$. Note that $o(a^k)=\frac{n}{gcd(k,n)}$ for all $k \in \{1,2,3...n-1\}$. Therefore, if $o(a) \neq o(a^k)$ for some $k \in  \{1,2,3...n-1\}$, then  $o(a^k)$ divides $o(a)$ and so $a$ and $a^k$ will become adjacent, which is not possible. Hence $o(a)=o(a^k)$ for all $k \in  \{1,2,3...n-1\}$. This implies that $gcd(k,n)=1$ for all $k \in  \{1,2,3...n-1\}$. Hence $n$ is prime.

Conversely, Suppose that $n$ is prime. Then $o(a^i)=o(a)=n$ for all $ i \in  \{1,2,3...n-1\}$. Moreover, $o(a^ib)=2$ for all $i \in  \{1,2,3...n-1\}$. Hence $OD(D_n)$ is a star graph.
\end{proof}

\begin{theorem}\label{2}
Let $G$ be a (finite) $p$-group of order $p^n$. Then order divisor graph $OD(G)$ is a complete multi-partite graph.
\end{theorem}
\begin{proof} Let $p$ be a prime and $A_i=\{ x \in G \mid o(x)=p^i\}$. We can write $G=A_0 \cup A_1 \cup \cdots \cup A_n$. It is not necessary that $G$ has elements of order $p^i$ for $i >1$ and by Cauchy's theorem, $G$ must have element of order $p$. Therefore, for some $i>1$,  $A_i$ may be empty but $A_0$ and $A_1$ must be nonempty in the union $A_0 \cup A_1 \cup \cdots \cup A_n$.
Note that $ab \not \in E(OD(G))$  if $a,b \in A_i$ for some $0\leq i \leq n$. Suppose $a \in A_i$ and $b \in A_j$, where $i \neq j$. Then $o(a)=p^i$ and $o(b)=p^j$. If $i < j$ then $p^i$ strictly divides $p^j$ and if $i > j$ then $p^j$ strictly divides $p^i$. Therefore $ab \in E(OD(G))$ and hence $OD(G)$ is complete multipartite graph.
\end{proof}

\begin{example}
\end{example}
$\\$
\begin{figure}[h!]
   \centering
     \includegraphics[width=0.4\textwidth]{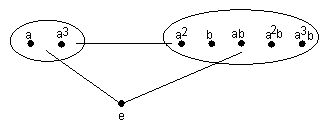}
     \caption{$OD(D_4)$}
     The line joining two independent sets of vertices means that each vertex in one independent set is adjacent to each vertex of other independent set.
\end{figure}
$\\$

\begin{corollary}\label{3}
Let $G$ be a finite cyclic group of order $p^n$.  Then $OD(G)$ is isomorphic to the complete $(n+1)$-partite graph $K_{1, p-1,  p(p-1), p^2(p-1) ,..., p^{n-1}(p-1)}$.
\end{corollary}
\begin{proof}Let $A_i=\{ x \in G : o(x)=p^i\}$. We can write $G=A_0 \cup A_1 \cup \cdots \cup A_n$. Since $G$ is cyclic, so by \cite[Theorem 4.4]{G}, $|A_i| = \phi(p^i)=p^{i-1}(p-1)$, where $\phi$ is {\em Euler's phi function}. Note that $ab \not \in E(OD(G))$  if $a,b \in A_i$ for any $0\leq i \leq n$. Suppose $a \in A_i$ and $b \in A_j$, where $i \neq j$. Then $o(a)=p^i$ and $o(b)=p^j$. If $i < j$ then $p^i$ strictly divides $p^j$ and if $i > j$ then $p^j$ strictly divides $p^i$. Therefore $ab \in E(OD(G))$ and hence $OD(G)$ is isomorphic to the complete $(n+1)$-partite graph $K_{1, p-1,  p(p-1), p^2(p-1) ,..., p^{n-1}(p-1)}$. \end{proof}
$\\$
\begin{corollary}\label{6}
If $G$ is a finite cyclic group of order $p^n$, then $\chi(OD(G))=n+1$.
\end{corollary}

\begin{example}The order divisor graph $OD(\mathbb{Z}_8)$ is shown below.
\end{example}
\begin{figure}[h!]
   \centering
     \includegraphics[width=0.5\textwidth]{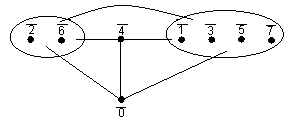}
     \caption{$OD(\mathbb{Z}_8)$}
     The line joining two independent sets of vertices means that each vertex in one independent set is adjacent to each vertex of other independent set.
\end{figure}

We denote by $G_1 \diamond G_2 \diamond \cdots \diamond G_k$  the sequential join of graphs $G_1,G_2,...,G_k$, where $G_i \diamond G_{i+1}=G_i \vee G_{i+1}$ for all $1 \leq i \leq k-1$, i.e., by adding an edge from each vertex of $G_i$ to each vertex of $G_{i+1}$, $1 \leq i \leq k-1$.

\begin{theorem}\label{5}
Let $G$ be a cyclic group of order $p_1p_2$, where $p_1, p_2$ are distinct primes. Then $OD(G)$ is a sequential join $(G_1 \diamond G_2 \diamond G_3)\diamond  K_1$,   where $G_1 \cong (p_1-1)K_1$,  $G_2 \cong (p_1-1)(p_2-1)K_1$ and $G_3 \cong (p_2-1)K_1$.
\end{theorem}
\begin{proof} The divisors of $p_1p_2$ are $1, p_1, p_2, p_1p_2$. We make a partition of the vertex set  $G$ as $G=A_{p_1} \cup A_{p_1p_2} \cup A_{p_2}\cup A_{1}$, where $A_{p_1}=\{ x \in G : o(x)=p_1\}$, $A_{p_1p_2}=\{ x \in G : o(x)=p_1p_2\}$,  $A_{p_2}=\{ x \in G : o(x)=p_2\}$ and $A_{1}=\{ x \in G : o(x)=1\}$. By \cite[Theorem 4.4]{G}, $|A_{p_1}| = \phi(p_1)=p_1 -1$, $|A_{p_1p_2}|= \phi(p_1p_2)=(p_1 -1)(p_2-1)$, $|A_{p_2}|= \phi(p_2)=(p_2-1)$ and $|A_{1}|=1$, where $\phi$ is the Euler's phi function. Since $p_1$ strictly divides $p_1p_2$, so each vertex in $A_{p_1}$ is adjacent to  each vertex in $A_{p_1p_2}$. Similarly, as $p_2$ strictly divides $p_1p_2$, so each vertex in $A_{p_2}$ is adjacent to each vertex in $A_{p_1p_2}$. Note that $p_1 \nmid p_2$ and $p_2 \nmid p_1$, so no vertex of $A_{p_1}$ is adjacent to any vertex of $A_{p_2}$ and no vertex of $A_{p_2}$ is adjacent to any vertex of $A_{p_1}$. Clearly, the single vertex in $A_1$, i.e  corresponding to the identity element of $G$,  is adjacent to every vertex in $A_{p_1}$, every element in $A_{p_1p_2}$ and every element in $A_{p_2}$. Also note that the vertices in $A_{p_1}$, $A_{p_1p_2}$ and $A_{p_2}$ are independent. Hence  $OD(G)=(G_1+G_2+G_3)+ K_1$, where $G_1 \cong (p_1-1)K_1$,  $G_2 \cong (p_1-1)(p_2-1)K_1$ and $G_3 \cong (p_2-1)K_1$.
\end{proof}
$\\$
\begin{example}
\end{example}
\begin{figure}[h!]
   \centering
     \includegraphics[width=0.5\textwidth]{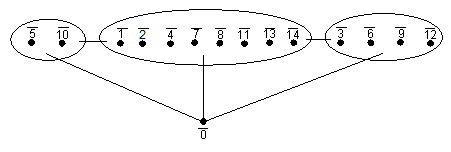}
     \caption{$OD(\mathbb{Z}_{15})$}
     The line joining two independent sets of vertices means that each vertex in one independent set is adjacent to each vertex of other independent set.
\end{figure}

\newpage

We denote by $G_1 \diamond G_2 \diamond \cdots \diamond G_k$  the sequential join of graphs $G_1,G_2,...,G_k$, where $G_i \diamond G_{i+1}=G_i \vee G_{i+1}$ for all $1 \leq i \leq k-1$, i.e., by adding an edge from each vertex of $G_i$ to each vertex of $G_{i+1}$, $1 \leq i \leq k-1$.

\begin{theorem}\label{7}
Let $H$ be a cyclic group of order $p_1p_2p_3$, where $p_1, p_2$ and $p_3$ are distinct primes. Then $OD(H)$ is  defined as:\\ $$\bigg ( \Big((G_1 \diamond G_{12})\cup(G_{12}\diamond G_{2}) \cup (G_2 \diamond G_{23}) \cup (G_{23} \diamond G_3) \cup (G_{3} \diamond G_{13})\cup (G_{13}\diamond G_{1}) \Big) \diamond G_{123}\bigg) \diamond K_1$$  where  $G_i \cong (p_i-1)K_1$, $1 \leq i \leq3 $, $G_{12} \cong (p_1-1)(p_2-1)K_1$, $G_{23}\cong (p_2-1)(p_3-1)K_1$, $G_{13}\cong (p_1-1)(p_3-1)K_1$, and $G_{123}\cong (p_1-1)(p_2-1)(p_3-1)K_1$.\\\
\end{theorem}

\begin{proof} Let $H$ be a cyclic group such that $|H|=p_1p_2p_3$, where $p_1, p_2$ and $p_3$ are distinct primes. The divisors of $p_1p_2p_3$ are $1, p_1, p_2, p_3, p_1p_2, p_1p_3, p_2p_3, p_1p_2p_3$. We make a partition of the vertex set $H$, based on the divisors of $p_1p_2p_3$, as follows: $H=A_{p_1} \cup A_{p_1p_2}\cup A_{p_1p_3} \cup A_{p_2}\cup A_{p_2p_3}\cup A_{p_3}\cup A_{p_1p_2p_3}\cup A_{1}$, where $A_{i}=\{ x \in H \mid o(x)=i\}$ for $ i \in \{1, p_1, p_2, p_3, p_1p_2, p_1p_3, p_2p_3, p_1p_2p_3\}$.  By \cite[Theorem 4.4]{G}, $|A_{p_1}| = \phi(p_1)=p_1 -1$, $|A_{p_1p_2}|= \phi(p_1p_2)=(p_1 -1)(p_2-1)$, $|A_{p_1p_3}|=\phi(p_1p_3)=(p_1-1)(p_3-1)$,$|A_{p_2}|= \phi(p_2)=(p_2-1)$, $|A_{p_2p_3}|= \phi(p_2p_3)=(p_2-1)(p_3-1)$, $|A_{p_3}|=\phi(p_3)=(p_3-1)$, $|A_{p_1p_2p_3}|= \phi(p_1p_2p_3)=(p_1 -1)(p_2-1)(p_3-1)$ and $|A_{1}|=1$, where $\phi$ is the Euler's phi function. Since $p_1$ strictly divides $p_1p_2$, $p_1p_3$ and $p_1p_2p_3$, so each vertex in $A_{p_1}$ is adjacent to each vertex in $A_{p_1p_2}$, $A_{p_1p_3}$ and $A_{p_1p_2p_3}$. As $p_2$ strictly divides $p_1p_2$, $p_2p_3$ and $p_1p_2p_3$, so each vertex in $A_{p_2}$ is adjacent to each vertex in $A_{p_1p_2}$, $A_{p_2p_3}$ and $A_{p_1p_2p_3}$. Similarly $p_3$ strictly divides $p_1p_3$, $p_2p_3$ and $p_1p_2p_3$, so each vertex in $A_{p_3}$ is adjacent to each vertex in $A_{p_1p_3}$, $A_{p_2p_3}$ and $A_{p_1p_2p_3}$. Note that $p_1 \nmid{p_2,p_3}$ , $p_2 \nmid {p_1,p_3}$ and $p_3 \nmid {p_1,p_2}$, so no vertex of $A_{p_1}$ is adjacent to any vertex of $A_{p_2}$ or $A_{p_3}$, no vertex of $A_{p_2}$ is adjacent to any vertex of $A_{p_1}$ or $A_{p_3}$ and similarly no vertex of $A_{p_3}$ is adjacent to any vertex of $A_{p_1}$ or $A_{p_2}$ . Clearly, the single vertex in $A_1$, i.e.,  corresponding to the identity element of $H$,  is adjacent to every vertex in $A_{p_1}$, $A_{p_2}$, $A_{p_3}$, every vertex in $A_{p_1p_2}$, $A_{p_1p_3}$, $A_{p_2p_3}$ and every vertex in $A_{p_1p_2p_3}$. Also note that the sets of vertices $A_{p_1}$, $A_{p_2}$, $A_{p_3}$, $A_{p_1p_2}$, $A_{p_1p_3}$, $A_{p_2p_3}$ and $A_{p_1p_2p_3}$ are independent sets. Hence $OD(H)$ is defined as $\bigg ( \Big((G_1 \diamond G_{12})\cup(G_{12}\diamond G_{2}) \cup (G_2 \diamond G_{23}) \cup (G_{23} \diamond G_3) \cup (G_{3} \diamond G_{13})\cup (G_{13}\diamond G_{1}) \Big) \diamond G_{123}\bigg) \diamond K_1$.
\end{proof}$\\$
\begin{example}
\end{example}
\begin{figure}[h!]
   \centering
     \includegraphics[width=0.7\textwidth]{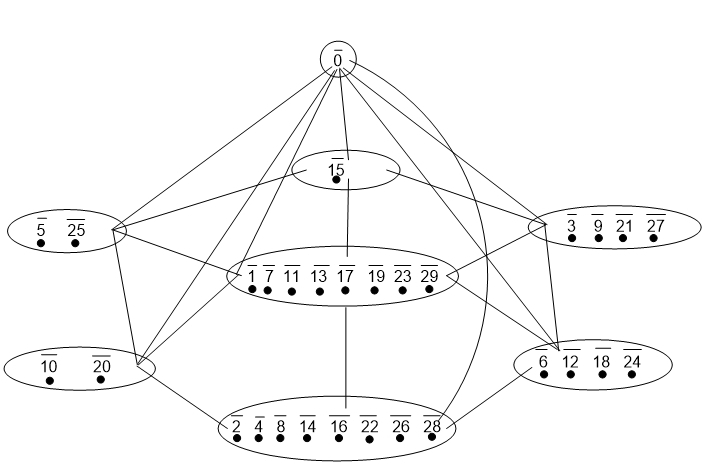}
     \caption{$OD(\mathbb{Z}_{30})$}
     The line joining two independent sets of vertices means that each vertex in one independent set is adjacent to each vertex of other independent set.
\end{figure}

\begin{definition}
 Let $n \in \mathbb{N}$ and let $D$ be the set of all $(positive)$ divisors of $n$. Define a partial order $\preceq$ on $D$ by $a \preceq b$ if
and only if $a \mid b$. Then $(D, \preceq)$ is a bounded lattice. We denote by $G_n$  the comparability graph on $(D, \preceq)$.
In other words, $G_n$ is a simple undirected graph with vertex set $D$ and two vertices $a$ and $b$ are
adjacent if and only if $a \neq  b$ and either $a \preceq b$ or $b \preceq a$. We denote by $\mathcal{E}(G_n)$ the extended graph of $G_n$ which is obtained by replacing each vertex $d$ in $G_n$ by $\phi(d)$ copies of $d$ which form an independent set.
\end{definition}

\begin{definition}
Given a graph $G= (V, E)$, the reduced graph of $G$, denoted by $\mathcal{R}(G)$,  is obtained from $G$ by merging those vertices which has  same
set of closed neighbors, where a closed neighbor of $v \in V$ is the set $\{u \in  V | uv \in E\} \cup\{v\}$.
\end{definition}

\begin{theorem}\label{8}
Let $G$ be a finite group of order $n$. The following are equivalent:
\begin{enumerate}
\item[{\em(a)}]  $G$ is cyclic;

\item[{\em(b)}]  $\mathcal{E}(G_n) \cong OD(G)$;

\item[{\em(c)}]  $G_n \cong \mathcal{R}(OD(G))$.
\end{enumerate}
\end{theorem}
\begin{proof}

$(a) \Leftrightarrow(b)$: Let $D$ be the set of all $(positive)$ divisors of $n$. Suppose $G$ is cyclic. Then for each $d \in D$,  $G$ has exactly $\phi(d)$ elements of order $d$, cf. \cite[Theorem 4.4]{G}. Hence by definitions of $OD(G)$ and $G_n$,  we get that $E(G_n) \cong OD(G)$. Conversely, suppose $\mathcal{E}(G_n) \cong OD(G)$. Since $n \in D$, so $G_n$ has a vertex associated to $n$ and hence $\mathcal{E}(G_n)$  has $\phi(n)$ vertices associated to the group elements of order $n$. Hence $G$ is cyclic.

$(a) \Leftrightarrow(c)$: Let $D$ be the set of all distinct $(positive)$ divisors of $n$. Suppose $G$ is cyclic. Then for each $d \in D$,  $G$ has exactly $\phi(d)$ elements of order $d$, cf. \cite[Theorem 4.4]{G}. Therefore, by definitions, $G_n$ is isomorphic to the reduced graph of $OD(G)$. Conversely, let $G_n \cong \mathcal{R}(OD(G))$. Since, $OD(G)$ consists of independent sets associated to elements of same order in $G$ and by reducing $OD(G)$ we obtain the comparability graph $G_n$  (in $G_n$ we have a vertex associated to each divisor of $n$), therefore in $OD(G)$ we have independent set of vertices associated to each divisor of  $|G|=n$. In particular, $G$ must have elements or order $n$. Hence $G$ is cyclic.

\end{proof}

\begin{remark}
We saw that for  finite cyclic groups $G$, the comparability graph $G_n$ can be studied by passing to the order divisor graph $OD(G)$. Also the order divisor graph $OD(G)$ can be studied by passing to the comparability graph $G_n$. That is, we have maps

                            $$\mathcal{E}$$
$$\{Comparability~graphs\}~\rightleftharpoons~\{Order~divisor~graphs\}$$
$$\mathcal{R}$$

\end{remark}

\end{document}